\documentclass[10pt,reqno]{amsart}
\usepackage[colorlinks=true,
pdfstartview=FitV, linkcolor=, citecolor=,
urlcolor=]{hyperref}
\usepackage{amsmath,amsfonts,latexsym,amssymb}
\usepackage{amssymb,amsfonts, amsmath}
\usepackage{mathrsfs}
\usepackage[latin1]{inputenc}
\usepackage[T1]{fontenc}
\usepackage{ae,aecompl}
\usepackage{braket}
\usepackage{comment}
\usepackage{amssymb,amsfonts, amsmath}
 \usepackage{mathtools} 
\usepackage[all,arc]{xy}
\usepackage{enumerate}
\usepackage{mathrsfs}
\usepackage{mathabx}
\usepackage{color}
\usepackage{graphicx}
\usepackage{subfig}
\usepackage{verbatim} 
\usepackage{amssymb,amsfonts, amsmath}
\usepackage[all,arc]{xy}
\usepackage{enumerate}
\usepackage{mathrsfs}
\usepackage{hyperref}
\usepackage{xstring}
\usepackage{amsmath}
\usepackage[makeroom]{cancel}

\makeatletter
\newsavebox{\@brx}
\newcommand{\llangle}[1][]{\savebox{\@brx}{\(\m@th{#1\langle}\)}%
  \mathopen{\copy\@brx\mkern2mu\kern-0.9\wd\@brx\usebox{\@brx}}}
\newcommand{\rrangle}[1][]{\savebox{\@brx}{\(\m@th{#1\rangle}\)}%
  \mathclose{\copy\@brx\mkern2mu\kern-0.9\wd\@brx\usebox{\@brx}}}
\makeatother

\newtheorem{theorem}{Theorem}[section]
\newtheorem{lemma}[theorem]{Lemma}
\newtheorem{proposition}[theorem]{Proposition}

\newtheorem{definition}[theorem]{Definition}

\newtheorem{fact}[theorem]{Fact}
\newtheorem{question}[theorem]{Question}

\renewcommand{\leq}{\leqslant}
\renewcommand{\geq}{\geqslant}

\long\def\@savemarbox#1#2{\global\setbox#1\vtop{\hsize\marginparwidth 
  \@parboxrestore\tiny\raggedright #2}}
\usepackage{braket}
\usepackage{amsmath}
\makeatletter
\renewcommand*\env@matrix[1][\arraystretch]{%
\edef\arraystretch{#1}%
\hskip -\arraycolsep
\let\@ifnextchar\new@ifnextchar
\array{*\c@MaxMatrixCols c}}
\makeatother
\date{\today}

\title{}
\date{\today}
\AtEndDocument{\bigskip{\footnotesize%
\textsc{CNRS, Laboratoire Alexander Grothendieck, Institut des Hautes \'Etudes Scientifiques, 
Universite Paris-Saclay, 35 route de Chartres, 91440 Bures-sur-Yvette, France } \par  
 \textit{E-mail address}: \texttt{tsouvkon@ihes.fr}}}

\author{Konstantinos Tsouvalas}
\begin{document}
\title{Fiber products of rank 1 superrigid lattices and quasi-isometric embeddings}
\frenchspacing 

\maketitle

\begin{abstract} Let $\Delta$ be a cocompact lattice in $\mathsf{Sp}(m,1)$, $m \geq 2$, or $\textup{F}_4^{(-20)}$. We exhibit examples of finitely generated subgroups of $\Delta \times \Delta$ with positive first Betti number all of whose discrete faithful representations into any real semisimple Lie group are quasi-isometric embeddings. The examples of this paper are inspired by the counterexamples of Bass--Lubotzky \cite{Bass-Lubotzky} to Platonov's conjecture. \end{abstract}
\section{Introduction}
\par Let $\Gamma$ be an irreducible lattice of a real algebraic semisimple Lie group $G$. If $\Gamma$ is cocompact, then the inclusion of $\Gamma$ in $G$ is a quasi-isometric embedding. More precisely, if we equip the symmetric space $G/K$ associated to $G$ with the left invariant Riemannian distance $d_{G/K}$ induced by the Killing metric, and identify (via the orbit map) $\Gamma$ as a subset of $G/K$, then $d_{G/K}$ restricted on $\Gamma$ is coarsely equivalent with any left invariant word metric on $\Gamma$ induced by a finite generating subset. If $\Gamma$ is not cocompact, then this might not be the case (e.g. for $G = \mathsf{SL}_2(\mathbb{R})$ and $\Gamma = \mathsf{SL}_2(\mathbb{Z})$). However, Lubotzky--Mozes--Raghunathan \cite{LMR} proved\footnote{More generally, the main result in \cite{LMR} applies to irreducible lattices in products $G_1(k_1)\times \cdots \times G_{\ell}(k_{\ell})$ of connected simple groups $G_i$ defined over a local field $k_i$.} that this is the case when the real rank of $G$ is at least $2$.
\par The superrigidity theorem of Margulis \cite[Thm. VII 5.6]{Mar} imposes severe restrictions on linear representations of the lattice $\Gamma$ over any local field. Following \cite{FH}, we say that a representation $\rho:\Gamma \rightarrow \mathsf{SL}_d(\mathbb{R})$ {\em almost extends to a continuous representation of $G$} if there exists a representation \hbox{$\hat{\rho}: G \rightarrow \mathsf{SL}_{d}(\mathbb{R})$} and a representation $\rho':\Gamma \rightarrow \mathsf{SL}_{d}(\mathbb{R})$ with compact closure such that the images of $\hat{\rho}(\Gamma)$ and $\rho'(\Gamma)$ commute and $\rho(\gamma)=\hat{\rho}(\gamma)\rho'(\gamma)$ for every $\gamma \in \Gamma$. Margulis' superrigidity theorem (see \cite[Thm. 3.8]{FH}) implies that every linear representation $\rho$ of the lattice $\Gamma$ almost extends to a continuous representation of $G$. In particular, if $G$ is simple and the image of $\rho$ has non-compact closure, then $\hat{\rho}$ has necessarily finite kernel and hence $\rho$ is a quasi-isometric embedding (see Definition \ref{qie}) since $\Gamma$ is quasi-isometrically embedded in $G$ by \cite{LMR}.
\par The supperrigidity of lattices in the quaternionic Lie group $\mathsf{Sp}(m,1)$, $m \geq 2$, or the rank 1 Lie group $\textup{F}_4^{(-20)}$ was established by the work of Corlette \cite{Corlette} and Gromov--Schoen \cite{GS}. Corlette's Archimedean supperrigidity, implies that every linear representation $\psi:\Delta \rightarrow \mathsf{SL}_d(\mathbb{R})$ of a lattice $\Delta$ in either $\mathsf{Sp}(m,1)$, $m \geq 2$, or $\textup{F}_4^{(-20)}$, almost extends to a continuous representation of $G$. In particular, if the image of $\psi$ has non-compact closure and $\Delta$ is cocompact, then $\psi$ is necessarily a quasi-isometric embedding.
\par The goal of this article is to exhibit examples of finitely generated linear groups which are not commensurable to a lattice in any semisimple Lie group and all of whose discrete and faithful linear representations into any real semisimple Lie group are quasi-isometric embeddings. For a group $\Lambda$ and a normal subgroup $L$ of $\Lambda$, {\em the fiber product of $\Lambda$ with respect to $L$} is the subgroup of the product $\Lambda \times \Lambda$ defined as follows $$\Lambda \times_{L}\Lambda =\big\{(gw,g): g \in \Lambda, w \in L\big\}.$$ We denote by $\mu:\mathsf{SL}_d(\mathbb{R}) \rightarrow \mathbb{R}^d$ the Cartan projection (see Section \ref{prelim}) and fix the usual Euclidean norm $||\cdot||_{\mathbb{E}}$ on $\mathbb{R}^d$. A representation $\psi: \mathsf{H} \rightarrow \mathsf{SL}_d(\mathbb{R})$ is called {\em distal}\footnote{The term {\em distal} is from \cite[p. 537]{DK}.} if for every $h \in \mathsf{H}$ the moduli of the eigenvalues of $\psi(h)$ are equal to $1$. For a group $N$ we denote by $[N,N]$ the \hbox{commutator subgroup of $N$.}

The main result of this article is the following:

\begin{theorem} \label{main1} Let $\Delta$ be a cocompact lattice in $\mathsf{Sp}(m,1)$, $m \geq 2$, or $\textup{F}_4^{(-20)}$. Suppose that $N$ is an infinite normal subgroup of $\Delta$ such that the quotient $\Delta/N$ is a non-elementary hyperbolic group. For a representation $\rho:\Delta \times_N \Delta \rightarrow \mathsf{SL}_d(\mathbb{R})$ the following conditions are equivalent:

\medskip
\noindent \textup{(i)} The restrictions of $\rho$ on $\{1\}\times [N,N]$ and $[N,N]\times \{1\}$ are not distal.

\noindent \textup{(ii)} $\rho$ is discrete and has finite kernel.

\noindent \textup{(iii)} $\rho$ is a quasi-isometric embedding.  \end{theorem}

 The existence of non-elementary hyperbolic quotients of the rank $1$ lattice $\Delta$ follows by the work of Gromov \cite{Gromov}, Olshanskii \cite{Ol} and Delzant \cite{Delzant}. M. Kapovich in  \cite[Thm. 8.1]{Kap} proved that any such quotient is a non-linear hyperbolic group. By using the Cohen--Lyndon theorem established in \cite{Sun}, similarly as in \cite{MM}, it is possible to exhibit infinitely many pairwise non-isomorphic fiber products $\Delta \times_{N} \Delta$ in $\Delta \times \Delta$ of positive first Betti number such that the quotient $\Delta/N$ is non elementary hyperbolic (see Proposition \ref{main3}). Moreover, the fiber product $\Delta \times_N \Delta$ cannot be commensurable to a lattice in any semisimple Lie group (see Proposition \ref{nonlattice}) and by Theorem \ref{main1} all discrete faithful representations of $\Delta\times_N \Delta$ into any semisimple Lie group are quasi-isometric embeddings.

Fiber products have been previously used (e.g. see \cite{PT,Bass-Lubotzky,Bridson-Grunewald}) in order to exhibit counterexamples in various settings. The examples of Theorem \ref{main1} are inspired by the construction of Bass--Lubotzky \cite{Bass-Lubotzky} of the first examples of finitely generated linear superrigid groups which are not commensurable with a lattice in any product $G_1(k_1)\times \cdots \times G_{\ell}(k_{\ell})$ of simple algebraic groups $G_i$ defined over a local field $k_i$. Their examples provide a negative answer to a conjecture of Platonov \cite[p. 437]{PR} and are constructed as the fiber product $\Lambda \times_L \Lambda$ of a cocompact lattice $\Lambda$ in the rank 1 Lie group $\textup{F}_4^{(-20)}$ with respect to a normal subgroup $L$ such that $\Lambda/L$ is a finitely presented group without non-trivial finite quotients and $\textup{H}_2(\Lambda/L,\mathbb{Z})=0$. Similar examples, where $\Lambda$ is a cocompact lattice in the quaternionic Lie group $\mathsf{Sp}(m,1)$, $m \geq 2$, were exhibited by Lubotzky \hbox{in \cite{Lubotzky}.} We would like to remark that is not clear whether the superrigid examples $\Lambda \times_L \Lambda$ of Bass--Lubotzky and Lubotzky can be chosen to admit a quasi-isometric embedding into some real semisimple Lie group. If this were the case, by \cite[Thm. 1.4 (b)]{Bass-Lubotzky}, $\Lambda \times_L \Lambda$ has to be a quasi-isometrically embedded subgroup of $\Lambda \times \Lambda$. In particular, since the distortion of $\Lambda \times_L \Lambda$ in the product $\Lambda \times \Lambda$ is linear, it follows that the quotient $\Lambda/L$ has linear Dehn function (e.g. see \cite[Prop. 3.2]{IT}). In partiuclar, it has to be a hyperbolic group which, by construction, does not admit non-trivial finite quotients. However, as of now, the existence of non-residually finite hyperbolic groups is an open problem and thus it is not clear whether $\Lambda/L$ can be chosen to be hyperbolic.

The proof of Theorem \ref{main1} is based on the following theorem. We use Corlette's Archimedean superrigidity \cite{Corlette} to prove:

\begin{theorem} \label{main2} Let $\Delta$ be a cocompact lattice in $\mathsf{Sp}(m,1)$, $m \geq 2$, or $\textup{F}_4^{(-20)}$. Fix \hbox{$|\cdot|_{\Delta}:\Delta \rightarrow \mathbb{N}$} a word length function on $\Delta$ and suppose that $N$ is an infinite normal subgroup of $\Delta$. Suppose that \hbox{$\rho:\Delta \times_{N} \Delta  \rightarrow \mathsf{SL}_d(\mathbb{R})$} is a representation such that the restrictions of $\rho$ on $\{1\}\times [N,N]$ and \hbox{$[N,N] \times \{1\}$} are not distal. Then there exist $C,c>0$ such that $$\big|\big| \mu\big(\rho(\delta n,\delta)\big) \big|\big|_{\mathbb{E}} \geq c\big(\big|\delta n \big|_{\Delta}+\big|\delta \big|_{\Delta}\big)-C$$ for every $\delta \in \Delta$ and $n \in N$.
\end{theorem}

Theorem \ref{main1} will follow by Theorem \ref{main2} and the fact that the fiber product $\Delta \times_N \Delta$ is a quasi-isometrically embedded subgroup of $\Delta \times \Delta$ when the quotient $\Delta/N$ is hyperbolic \hbox{(see Proposition \ref{undist}).}

We close this section by raising the following question which motivated the construction of the examples of this article. We only consider linear representations over $\mathbb{R}$ or $\mathbb{C}$.

\begin{question} Does there exist a finitely generated group $\mathsf{\Gamma}$ which is not commensurable to a lattice in any connected semisimple Lie group, admits a discrete faithful linear representation and every linear representation of $\mathsf{\Gamma}$ with non-compact closure is a quasi-isometric embedding?\end{question}

The fiber products provided by Theorem \ref{main1} admit representations with infinite kernel whose image has non-compact closure, hence they do not provide a positive answer to the previous question.

The paper is organized as follows. In Section \ref{prelim} we provide some background on fiber products, proximality and semisimple representations. In Section \ref{proofs} we prove Theorem \ref{main1} and Theorem \ref{main2}. Finally, in Section \ref{add}, we provide some additional properties of the examples of Theorem \ref{main1}.
 \medskip

\noindent \textbf{Acknowledgements.} I would like to thank Fanny Kassel for fruitful discussions and comments on an earlier version of this paper. I would also like to thank the referee for carefully reading the paper and for their comments and suggestions. This work received funding from the European Research Council (ERC) under the European's Union Horizon 2020 research and innovation programme (ERC starting grant DiGGeS, grant agreement No 715982).

\section{Preliminaries} \label{prelim}
Let $\Gamma$ be a finitely generated group. Throughout this paper, we fix a left invariant word metric $d_{\Gamma}:\Gamma \times \Gamma \rightarrow \mathbb{N}$ induced by some finite generating subset of $\Gamma$. The word length fuction $|\cdot|_{\Gamma}:\Gamma \rightarrow \mathbb{N}$ is defined as $|\gamma|_{\Gamma}=d_{\Gamma}(\gamma,1)$, where $1 \in \Gamma$ denotes the identity element of $\Gamma$. The {\em normal closure of a subset $\mathcal{F}$} of $\Gamma$ is the normal subgroup $\llangle \mathcal{F} \rrangle=\langle \{ gfg^{-1}: g \in \Gamma,f \in \mathcal{F}\} \rangle$ of $\Gamma$. For two subgroups $A,B$ of $\Gamma$ the commutator $[A,B]$ is the group $[A,B]=\llangle \{ [g,h]:g\in A, h \in B\} \rrangle$. The group $\Gamma$ is called {\em  \textup{(}Gromov\textup{)} hyperbolic} if the Cayley graph of $\Gamma$ with respect to a fixed generating set  is a $\delta$-Gromov hyperbolic space for some $\delta \geq 0$ (see \cite{Gromov}). A hyperbolic group is called {\em elementary} if it is virtually cyclic.
\subsection{Cartan and Jordan projection} For a matrix $g \in \mathsf{SL}_d(\mathbb{R})$ we denote by $\ell_1(g)\geq \ldots \geq \ell_d(g)$ (resp. $\sigma_1(g)\geq \ldots \geq \sigma_d(g)$) the moduli of the eigenvalues (resp. singular values) of $g$ in non-increasing order. The {\em Cartan projection} $\mu:\mathsf{SL}_d(\mathbb{R})\rightarrow \mathbb{R}^d$ is the map $$\mu(g)=\big(\log\sigma_1(g),\ldots,\log \sigma_d(g) \big),$$ for $g \in \mathsf{SL}_d(\mathbb{R})$. The Cartan projection $\mu$ is a continuous, proper and surjective onto $\{(x_1,\ldots, x_d):x_1\geq \cdots \geq x_d\}$. We denote by $||\cdot ||$ the standard operator norm where $\big|\big|g\big|\big|=\sigma_1(g)$ for every $g \in \mathsf{SL}_d(\mathbb{R})$. The {\em Jordan projection} is the map $\lambda:\mathsf{SL}_d(\mathbb{R})\rightarrow \mathbb{R}^d$ is the map defined as follows $$\lambda(g)=\big(\log\ell_1(g), \ldots, \log \ell_d(g)\big)$$ for $g \in \mathsf{SL}_d(\mathbb{R})$. The Cartan and Jordan projection are related as follows $$\lambda(g)=\lim_{r \rightarrow \infty}\frac{1}{r}\mu(g^r),$$ for $g \in \mathsf{SL}_d(\mathbb{R})$ (e.g. see \cite{benoist-limitcone}).

\begin{definition}\label{qie} Let $\Gamma$ be a finitely generated group. A representation $\rho:\Gamma \rightarrow \mathsf{SL}_d(\mathbb{R})$ is called a quasi-isometric embedding if there exist $C,c>1$ such that for every $\gamma \in \Gamma$ we have $$C^{-1}\big|\gamma \big|_{\Gamma}-c \leq \big| \big|\mu (\rho(\gamma) ) \big|\big|_{\mathbb{E}} \leq C\big|\gamma \big|_{\Gamma}+c.$$ \end{definition}

\noindent Equivalently, if we equip the symmetric space $\mathsf{X}_{d}=\mathsf{SL}_d(\mathbb{R})/K_d$, where $K_d=\mathsf{SO}(d)$, with the distance function $$\mathsf{d}\big(gK_d,hK_d\big)=\Big(\sum_{i=1}^{d} \big(\log \sigma_i (g^{-1}h)\big)^2 \Big)^{\frac{1}{2}} \ \ g,h\in \mathsf{SL}_{d}(\mathbb{R}),$$ $\rho$ is a quasi-isometric embedding if and only if its map $\tau_{\rho}:(\Gamma,d_{\Gamma}) \rightarrow (\mathsf{X}_{d}, \mathsf{d})$, $\tau_{\rho}(\gamma)=\rho(\gamma)K_{d}$ for $\gamma \in \Gamma,$ is a quasi-isometric embedding.

We also need the following elementary fact for the Cartan projection of a matrix $g \in \mathsf{SL}_d(\mathbb{R})$ and its exterior powers.

\begin{fact} \label{exterior} Let $d \geq 2$ and $1 \leq m \leq d-1$. There exists a constant $C_{d,m}>1$, depending only on $d,m \in \mathbb{N}$ such that for every $g \in \mathsf{SL}_d(\mathbb{R})$ we have $$C_{d,m}^{-1}\big|\big| \mu(g)\big|\big|_{\mathbb{E}} \leq\big|\big| \mu(\wedge^m g)\big|\big|_{\mathbb{E}} \leq C_{d,m}\big|\big| \mu(g)\big|\big|_{\mathbb{E}}.$$ In particular, for a finitely generated group $\Gamma$, a representation $\rho: \Gamma \rightarrow \mathsf{SL}_d(\mathbb{R})$ is a quasi-isometric embedding if and only if $\wedge^m \rho:\Gamma \rightarrow \mathsf{SL}(\wedge^m \mathbb{R}^d)$ is a quasi-isometric embedding.\end{fact}

We give here a proof for the reader's convenience.

\begin{proof} Note that for every $h\in \mathsf{SL}_{p}(\mathbb{R})$ and $1\leq i\leq p$, we have $\log \frac{\sigma_1(h)}{\sigma_p(h)}\geq |\log \sigma_i(h)|$ and $2\big((\log \sigma_1(h))^2+(\log \sigma_p(h))^2\big)\geq \big(\log \sigma_1(h)-\log\sigma_p(h)\big)^2$, so we obtain the double inequality $$\frac{1}{\sqrt{2}}\log \frac{\sigma_1(h)}{\sigma_p(h)}\leq \big|\big| \mu(h)\big|\big|_{\mathbb{E}} \leq \sqrt{p}\log \frac{\sigma_1(h)}{\sigma_p(h)}.$$ Let us set $r_{d,m}:=\binom{d}{m}$. For $g\in \mathsf{SL}_d(\mathbb{R})$ we have $$\sigma_1(\wedge^m g)=\sigma_1(g)\cdots \sigma_m(g), \ \sigma_{r_{d,m}}(\wedge^m g)=\sigma_d(g)\cdots \sigma_{d-m+1}(g)$$ and hence the previous bound shows that \begin{align*} \frac{1}{\sqrt{2}}\log \frac{\sigma_1(g)}{\sigma_d(g)}\leq \big|\big| \mu(\wedge^m g)\big|\big|_{\mathbb{E}} &\leq \sqrt{r_{d,m}}\log \frac{\sigma_1(g)\cdots \sigma_m(g)}{\sigma_d(g)\cdots \sigma_{d-m+1}(g)}\\ &\leq m\sqrt{r_{d,m}} \log\frac{\sigma_1(g)}{\sigma_d(g)}\leq m\sqrt{2r_{d,m}} \big|\big|\mu(g)\big|\big|_{\mathbb{E}}.\end{align*}  The inequality follows by taking $C_{d,m}:=m\sqrt{2r_{d,m}}$.\end{proof}

\subsection{Fiber products} Let $\Delta$ be a group and $N$ be a normal subgroup of $\Delta$. Let us recall that the {\em fiber product of $\Gamma$ with respect to $N$} is the subgroup of $\Delta \times \Delta$ generated by the diagonal subgroup $\textup{diag}(\Delta \times \Delta)=\big\{(\delta,\delta):\delta\in \Delta\big \}$ and the subgroup $N \times \{1\}$: $$\Delta \times_N\Delta=\big\{(\delta n, \delta):\delta \in \Delta,  n \in N \big\}.$$

\noindent Suppose that $\Delta$ is finitely generated, $S$ is a finite generating subset of $\Delta$ and $N=\llangle \mathcal{F} \rrangle$ for some finite subset $\mathcal{F}$ of $\Delta$. Observe that $\big \{(s,s):s \in S \big\} \cup \big \{(1,w):w \in \mathcal{F}\big\}$ is a finite generating subset of $\Delta\times_N \Delta$. In the special case where the quotient group $\Delta/N$ is hyperbolic, the fiber product $\Delta \times_N \Delta$ is an undistorted subgroup of (e.g. quasi-isometrically embedded) $\Delta \times \Delta$. This fact will follow by \cite[Thm. 2]{OS}.

\begin{proposition} \label{undist} Suppose that $\Delta=\langle X | R \rangle$ is a finitely presented group and $N$ is a normal subgroup of $\Delta$ such that $\Delta/N$ is hyperbolic. Let us fix a word length function $|\cdot|_{\Delta \times_N \Delta}:\Delta \times_N \Delta \rightarrow \mathbb{N}$. Then there exist $C,c>0$ such that for every $\delta \in \Delta$ and $n \in N$ we have  $$\big|(\delta n,\delta) \big|_{\Delta \times_N \Delta} \leq C\big( \big|\delta n|_{\Delta}+\big|\delta \big|_{\Delta}\big)+c.$$\end{proposition}

\begin{proof} We remark that since $\Delta/N$ is finitely presented, there exists a finite subset $\mathcal{F}$ of $\Delta$ such that $N=\llangle \mathcal{F}\rrangle$. Let $F(X)$ be the free group on the set $X$ and denote by $\pi: F(X) \twoheadrightarrow \Delta$ the projection onto $\Delta$ with kernel $\llangle R \rrangle$. Let $\overline{\mathcal{F}}$ be a finite subset of $F(X)$ with $\pi(\overline{\mathcal{F}})=\mathcal{F}$ and $\overline{N}=\llangle R \cup \overline{\mathcal{F}} \rrangle$. Note that the product $\pi \times \pi: F(X)\times F(X)\twoheadrightarrow  \Delta \times \Delta$ restricts to an epimorphism $\pi \times \pi: F(X)\times_{\overline{N}} F(X)\twoheadrightarrow  \Delta \times_{N} \Delta$. Since $\Delta/N=F(X)/\overline{N}$ is hyberbolic, by \cite[Thm. 2]{OS} it follows that $F(X)\times_{\overline{N}}F(X)$ is quasi-isometrically embedded in $F(X)\times F(X)$ i.e. there exist $C_0,c_0>0$ such that \begin{equation} \label{undist-eq1} \big|\big(\pi(g),\pi(gn)\big) \big|_{\Delta \times_{N}\Delta} \leq C_0 \big (|g|_{F(X)}+|n|_{F(X)}\big)+c_0 \end{equation} for every $g \in F(X)$ and $n \in \overline{N}$. Now the conclusion follows by (\ref{undist-eq1}) and the observation that for every $\delta \in \Delta$ there exists $\overline{\delta}\in F(X)$ with $\delta=\pi(\overline{\delta})$ and \hbox{$|\overline{\delta}|_{F(X)}\leq  |\delta|_{\Delta}$.}\end{proof}

\subsection{Proximality} An element $g \in \mathsf{SL}_d(\mathbb{R})$ is called {\em proximal} if $\ell_1(g)>\ell_2(g)$. In this case $g$ has a unique eigenvalue of maximum modulus. In addition, $g$ admits a unique attracting fixed point $x_g^{+}$ in $\mathbb{P}(V)$ and a repelling hypeplane $V_{g}^{-}$ such that $V=x_{g}^{+}\oplus V_{g}^{-}$ and for every $y \in \mathbb{P}(V)\smallsetminus \mathbb{P}(V_{g}^{-})$, $\lim_{n}g^ny=x_{g}^{+}$. 

For a subgroup $N$ of $\mathsf{SL}_d(\mathbb{R})$ containing a proximal element, the proximal limit set of $N$, denoted by $\Lambda_{N}^{\mathbb{P}}$, is the closure of the attracting fixed points in $\mathbb{P}(\mathbb{R}^d)$ of all proximal elements in $N$. In the special case where $N$ is irreducible (i.e. does not preserve any proper subspace of $\mathbb{R}^d$) we have the following fact proved by Benoist.

\begin{lemma}\textup{(}\cite{benoist-limitcone}\textup{)} \label{minimal} Let $N$ be an irreducible subgroup of $\mathsf{SL}_d(\mathbb{R})$ which contains a proximal element. Then $N$ acts minimally on $\Lambda_{N}^{\mathbb{P}}$ in $\mathbb{P}(\mathbb{R}^d)$ \textup{(}i.e. $\overline{N\cdot x}=\Lambda_{N}^{\mathbb{P}}$ for every \hbox{$x \in \Lambda_{N}^{\mathbb{P}}$\textup{)}.} \end{lemma}

\subsection{Semisimple representations} Let $\Gamma$ be a discrete group. A representation $\rho:\Gamma \rightarrow \mathsf{GL}_d(\mathbb{R})$ is called {\em semisimple} if $\rho$ decomposes as a direct sum of irreducible representations of $\Gamma$. In this case, the Zariski closure $\overline{\rho(\Gamma)}^{\textup{Zar}}$ of $\rho(\Gamma)$ in $\mathsf{GL}_{d}(\mathbb{R})$ is a real reductive algebraic Lie group. Moreover, every linear representation over $\mathbb{R}$ of $\overline{\rho(\Gamma)}^{\textup{Zar}}$ is semisimple, see \cite[Ch.3, Thm. 3.13.1]{Var}. In particular, if $\rho:\Gamma \rightarrow \mathsf{GL}_d(\mathbb{R})$ is semisimple, all  of its exterior powers $\wedge^{i}\rho$, $1\leq i\leq d-1$, are also semisimple. By default, we consider the trivial representation as semisimple.
\par Let $\psi:\Gamma \rightarrow \mathsf{GL}_d(\mathbb{R})$ be a representation and $G$ be the Zariski closure of $\psi(\Gamma)$ in $\mathsf{GL}_d(\mathbb{R})$. Let us choose a Levi decomposition $G=L \ltimes U_{G}$, where $U_{G}$ is the unipotent radical of $G$ and denote by $\pi:G\rightarrow L$ the canonical projection. The {\em semisimplification \hbox{$\psi^{ss}:\Gamma \rightarrow \mathsf{GL}_d(\mathbb{R})$} of $\psi$} is the composition $\psi^{ss}=\psi\circ \pi$. The semisimplification $\psi^{ss}$ does not depend on the choice of $L$ up to conjugation by an element of $U_{G}$ (see \cite[p. 24]{GGKW} for more details). One of the key properties of the semisimplification $\psi^{ss}$ is that $\rho^{ss}$ is a limit of conjugates of $\rho$ and hence $$\lambda(\psi(\gamma))=\lambda (\psi^{ss}(\gamma)) \ \ \forall \gamma \in \Gamma.$$

The following result, was established by Benoist in \cite{benoist-limitcone} by using a result of Abels-Margulis-Soifer \cite{AMS} (see also \cite[Thm. 4.12]{GGKW} for a proof) and offers a connection between eigenvalues and singular values of elements in a semisimple subgroup of $\mathsf{SL}_d(\mathbb{R})$.

\begin{theorem} \textup{(\cite{AMS, benoist-limitcone})} \label{finitesubset} Let $\Gamma$ be an abstract group and $\rho:\Gamma \rightarrow \mathsf{SL}_{d}(\mathbb{R})$ be a semisimple representation. Then there exists a finite subset $F$ of $\Gamma$ and $C_{\rho}>0$ with the property: for every $\gamma \in \Gamma$ there exists $f \in F$ such that $$\max_{1 \leq i \leq d} \big| \log \sigma_i(\rho(\gamma))-\log \ell_i(\rho(\gamma f))\big|\leq C_{\rho}.$$ \end{theorem}

We also need the following lemma which essentially follows by the previous theorem.

\begin{lemma} \label{qie-semisimple} Let $\Gamma$ be a finitely generated group, $\rho:\Gamma \rightarrow \mathsf{SL}_d(\mathbb{R})$ be a representation and \hbox{$\rho^{ss}:\Gamma \rightarrow \mathsf{SL}_d(\mathbb{R})$} be a semisimplification of $\rho$. There exists $\delta>0$ such that $$\big|\big|\rho(\gamma)\big|\big| \geq \delta \big|\big|\rho^{ss}(\gamma)\big|\big| $$ for every $\gamma \in \Gamma$. In particular, if $\rho^{ss}$ is a quasi-isometric embedding then $\rho$ is also a quasi-isometric embedding.\end{lemma}

\begin{proof} Since $\rho^{ss}$ is semisimple and $\lambda(\rho^{ss}(g))=\lambda(\rho(g))$ for every $g \in \Gamma$, by Theorem \ref{finitesubset}, there exists a finite subset $F$ of $\Gamma$ and $C>0$ with the property: for every $\gamma \in \Gamma$ there exists $f \in F$ with \begin{align}\label{AMSineq}\big| \log \ell_1(\rho(\gamma f))-\log \sigma_1(\rho^{ss}(\gamma))\big|\leq C\end{align} Now choose $\gamma \in \Gamma$ and $f\in F$ satisfying (\ref{AMSineq}). By the submultiplicativity of the operator norm and the fact that $||g||\geq \ell_1(g)$ for $g\in \mathsf{SL}_d(\mathbb{R})$, we conclude for every $\gamma \in \Gamma$ that \begin{align*} \big|\big|\rho(\gamma)\big|\big| \geq \big|\big|\rho(f)\big|\big|^{-1} \big|\big|\rho(\gamma f)\big|\big|\geq \big|\big|\rho(f)\big|\big|^{-1} \ell_1(\rho(\gamma f)) \geq \big|\big|\rho(f)\big|\big|^{-1}e^{-C} \big|\big| \rho^{ss}(\gamma)\big|\big|.\end{align*} In particular, we conclude that $$  \big|\big|\rho(\gamma)\big|\big| \geq \big(\min_{f\in F} \big|\big|\rho(f)\big|\big|^{-1}\big)e^{-C} \big|\big| \rho^{ss}(\gamma)\big|\big|$$ for every $\gamma \in \Gamma$. The inequality follows.\end{proof}

The following lemma is a consequence of Margulis' lemma.

\begin{lemma} \label{non-distal} Suppose that $\mathsf{H}$ is a group which contains a free subgroup $F\subset \mathsf{H}$ of rank $2$ and \hbox{$\psi:\mathsf{H}\rightarrow \mathsf{SL}_d(\mathbb{R})$} is a discrete faithful representation. Then $\psi(\mathsf{H})$ is not distal. \end{lemma}

\begin{proof} Let $\psi_0$ be a semisimplification of the restriction $\psi|_{F}:F \rightarrow \mathsf{SL}_d(\mathbb{R})$. Since $\psi$ is discrete and faithful, $\psi_0$ is an algebriac limit of discrete faithful representations, namely of conjugates of $\psi|_{F}$. Since $F$ does not contain non-trivial normal nilpotent subgroups, it follows by \cite[Thm. 2.11]{BIW} that $\psi_0$ is discrete and faithful. In particular, by Theorem \ref{finitesubset}, $\psi_0$  cannot be distal. Therefore, since $\lambda(\psi(g))=\lambda(\psi_0(g))$ for every $g\in F$, we conclude that $\psi|_F$ is not distal. \end{proof}

We will also need the following well known elementary fact.

\begin{fact}\label{normal} Let $\Gamma$ be a group and $N$ be a normal subgroup of $\Gamma$. Suppose that $\rho:\Gamma \rightarrow \mathsf{GL}_d(\mathbb{R})$ is a semisimple representation. Then the restriction $\rho|_N:N\rightarrow \mathsf{GL}_d(\mathbb{R})$ of $\rho$ on $N$ is semisimple.\end{fact}

\begin{proof} We may assume that $\rho|_N$ is non-trivial and $\rho$ is irreducible. Let $V_0\neq \{0\}$ be an $\rho(N)$-invariant subspace of $\mathbb{R}^d$ of minimal dimension. Since $\textup{dim}(V_0)$ is minimal, $\rho(N)$ acts irreducibly  on $V_0$. We claim that there exist $h_1\ldots, h_r\in \Gamma$ such that $V=\oplus_{i=1}^{r} \rho(h_i)V_0$. This is enough to conclude that $\rho|_N$ is semisimple since $N$ is normal in $\Gamma$ and $\rho(N)$ preserves and acts irreducibly on $\rho(h_i)V_0$ for each $i$.

We prove the claim. If $V_0=\mathbb{R}^d$ then $\rho|_N$ is obviously semisimple and we take $h_1=1$. If $V_0$ is a proper subspace of $\mathbb{R}^d$, there exists $h_2 \in \Gamma$ such that $\rho(h_2)V_0\neq V_0$. Then $\rho(h_2)V_0\cap V_0$ is a proper subspace of $V_0$ which is $\rho(N)$-invariant and hence $\rho(h_2)V_0\cap V_0=\{0\}$. It follows that $V_1:=V_0+\rho(h_2)V_0$ is direct. If $V_1=\mathbb{R}^d$, then $\rho|_{N}$ is semisimple with two irreducible components. Otherwise, there exists $h_3 \in \Gamma$ such that $\rho(h_3)V_0$ is not a subspace of $V_1$. In particular, $\rho(h_3)V_0\cap V_1$ is a proper $\rho(N)$-invariant subspace of $V_0$ and so $\rho(h_3)V_0\cap V_1=\{0\}$. It follows that $\rho(h_3)V_0+V_1$ is direct and $\rho|_{N}$ is semisimple. By continuing similarly we obtain the conclusion. \end{proof}

\subsection{Archimedean superrigidity} We review here the following version of Corlette's Archimedean superrigidity theorem \cite{Corlette} (see \cite[Thm. 3.8]{FH}) that we use for the proof of Theorem \ref{main2}. 

\begin{theorem} \label{superrigidity} Let $G$ be either $\mathsf{Sp}(m,1)$, $m \geq 2$, or $\textup{F}_4^{(-20)}$ and $\Delta$ be a lattice in $G$. Suppose that $\rho:\Gamma \rightarrow \mathsf{SL}_d(\mathbb{R})$ is a representation. Then there exists a continuous representation $\hat{\rho}:G \rightarrow \mathsf{SL}_d(\mathbb{R})$ and a representation $\rho':\Delta \rightarrow \mathsf{GL}_d(\mathbb{R})$ with compact Zariski closure such that:
\medskip

\noindent \textup{(i)} the images $\hat{\rho}(\Delta)$ and $\rho'(\Delta)$ commute.\\
\noindent \textup{(ii)} $\rho(\delta)=\hat{\rho}(\delta)\rho'(\delta)$ for every $\delta \in \Delta$.\end{theorem}

\begin{rmk} Let us recall that when $\rho$ has non-compact closure then $\hat{\rho}$ has necessarily finite kernel and there exists $C_{\rho}>0$ such that $\big|\big| \mu (\rho(\delta))-\mu (\hat{\rho}(\delta)) \big| \big|_{\mathbb{E}}\leq C_{\rho}$ for every $\delta \in \Delta$. In particular, if $\Delta$ is cocompact in $G$ then $\rho$ is a quasi-isometric embedding.\end{rmk}

\section{Proof of Theorem \ref{main1} and Theorem \ref{main2}} \label{proofs}

We first prove Theorem \ref{main2} which we use for the proof of Theorem \ref{main1}. In order to simplify our notation, for a subgroup $\mathsf{H}$ of $\Delta$ we consider the subgroups of $\Delta \times \Delta$: $$\mathsf{H}_{\textup{L}}:=\mathsf{H} \times \{1\} \ \ \textup{and} \ \ \mathsf{H}_{\textup{R}}:=\{1\}\times \mathsf{H}.$$

\begin{proof}[Proof of Theorem \ref{main2}.] Let us set $\Gamma:=\Delta \times_N\Delta$. There are two cases to consider.
\medskip

\noindent {\bf Case 1}: {\em The representation $\rho:\Gamma \rightarrow \mathsf{SL}_d(\mathbb{R})$ is semisimple.} 

\begin{proof}[Proof of Case 1.] Since the restriction of $\rho$ on $[N,N]_{\textup{R}}$ is not distal, there exists $w_{0}\in [N,N]$ and $1 \leq r \leq d-1$ such that $\ell_{r}(\rho(1,w_0))>\ell_{r+1}(\rho(1,w_0))$ and $\wedge^{r}\rho(1,w_0)$ is proximal.

Let us consider the exterior power $\psi:\Gamma \rightarrow \mathsf{SL}(\wedge^{r}\mathbb{R}^d)$, $\psi:=\wedge^{r}\rho$, of $\rho$. The representation $\psi$ is semisimple and the proximal limit set $\Lambda_{\psi(N_{\textup{R}})}^{\mathbb{P}}$ in $\mathbb{P}(\wedge^{r}\mathbb{R}^d)$ is non-empty. Let $V$ be the vector subspace of $\wedge^r \mathbb{R}^d$ spanned by the attracting eigenlines of proximal elements of $\psi(N_{\textup{R}})$.  Since $\psi|_{N_{\textup{R}}}$ is semisimple by Fact \ref{normal} (and hence its restriction on $V$), there exists a decomposition $$V=V_{1}\oplus \cdots \oplus V_{\ell}$$ such that $\psi(N_{\textup{R}})$ preserves and acts irreducibly on each $V_{i}$, $1 \leq i \leq \ell$. Note also that the restriction of $\psi_i:N_{\textup{R}}\rightarrow \mathsf{GL}(V_i)$ of $\psi|_{N_{\textup{R}}}$ on $V_i$ is proximal by the definition of $V$.

Since $\psi(N_{\textup{L}})$ centralizes $\psi(N_{\textup{R}})$, $\psi(N_{\textup{L}})$ fixes pointwise the attracting eigenline of every proximal element of $\psi(N_{\textup{R}})$. In particular, $\psi(N_{\textup{L}})$ fixes pointwise a basis of each of the subspaces $V_{1},\ldots, V_{\ell}$ of $V$. Since $\psi(N_{\textup{R}})|_{V_i}$ is irreducible, every element in $\psi(N_{\textup{L}})$ acts as a scalar on each $V_{i}$. In other words, there exist finitely many group homomorphisms \hbox{$\varepsilon_1,\ldots, \varepsilon_{\ell}:N \rightarrow \mathbb{R}^{\ast}$} such that \begin{equation} \label{thm2-eq1} \psi(n,1)v_i=\varepsilon_i(n)v_i \ \ \forall v_i\in V_i \ \forall n \in N. \end{equation} For $1 \leq i \leq \ell$, let $\Lambda_{\psi_i}^{\mathbb{P}}$ be the proximal limit set of $\psi_{i}$ in $\mathbb{P}(V_i)$ and note that $\Lambda_{\psi(N_{\textup{R}})}^{\mathbb{P}}=\Lambda_{\psi_i}^{\mathbb{P}} \cup \cdots \cup \Lambda_{\psi_{\ell}}^{\mathbb{P}}$. 
\medskip

We need the following claim.
\medskip

\noindent {\em Claim 1.} {\em For every $\delta\in \Delta$, $\psi(\delta,\delta)$ determines a bijection of $\sigma(\delta):\{1,\ldots,\ell \}\rightarrow \{1,\ldots,\ell \}$ as follows: $$\psi(\delta,\delta)\Lambda_{\psi_{i}}^{\mathbb{P}}=\Lambda_{\psi_{\sigma(i)}}^{\mathbb{P}}, \ 1\leq i \leq \ell.$$} To verify the previous claim it suffices to check that if $\psi(\delta,\delta)\Lambda_{\psi}^{\mathbb{P}}\cap \Lambda_{\psi_j}^{\mathbb{P}}$ is non-empty, then $\psi(\delta,\delta)\Lambda_{\psi_i}^{\mathbb{P}}=\Lambda_{\psi_j}^{\mathbb{P}}$. Suppose that $x_{i}\in \Lambda_{\psi_i}^{\mathbb{P}}$ with $\rho(\delta,\delta)x_{i}\in \Lambda_{\psi_j}^{\mathbb{P}}$. Since $\Lambda_{\psi_j}^{\mathbb{P}}$ is $\psi(N_{\textup{R}})$-invariant, for every $n \in N$ we have $$\rho(1,\delta n \delta^{-1})\rho(\delta,\delta)x_{i}=\rho(\delta,\delta)\rho(1,n)x_{i} \in \Lambda_{\psi_j}^{\mathbb{P}}.$$ In particular, since $\Lambda_{\psi_j}^{\mathbb{P}}$ is closed, we have $\psi(\delta,\delta)\overline{\psi(N_{\textup{R}})x_{i}}\subset \Lambda_{\psi_j}^{\mathbb{P}}$. Since $\psi(N_{\textup{R}})$ acts irreducibly on $V_{i}$, by Lemma \ref{minimal}, it acts minimally on $\Lambda_{\psi_i}^{\mathbb{P}}$ and hence $\psi(\delta,\delta)\Lambda_{\psi_i}^{\mathbb{P}}\subset \Lambda_{\psi_{j}}^{\mathbb{P}}$. Since $\Lambda_{\psi_i}^{\mathbb{P}}\cap \psi(\delta,\delta)^{-1}\Lambda_{\psi_{j}}^{\mathbb{P}}$ is non-empty we similarly deduce that $ \psi(\delta,\delta)^{-1}\Lambda_{\psi_{j}}^{\mathbb{P}}\subset \Lambda_{\psi_i}^{\mathbb{P}}$.

The claim follows. We conclude that $\psi(\delta,\delta)\Lambda_{\psi_i}^{\mathbb{P}}=\Lambda_{\psi_j}^{\mathbb{P}}$ and hence $\sigma(\delta)$ is a well defined bijection of $\{1,\ldots,\ell \}$. Finally, we obtain a well defined group homomorphism $\sigma:\Delta \rightarrow \textup{Sym}\big(\{1,\ldots,\ell \}\big)$, $\delta \mapsto \sigma(\delta)$. The group $\Delta':=\textup{ker}\sigma$ is a finite index subgroup $\Delta$ (of index at most $\ell!$) and by the definition of $\sigma$ has the property that \begin{equation} \label{thm2-eq0} \psi(\delta,\delta)\Lambda_{\psi_i}^{\mathbb{P}}=\Lambda_{\psi_i}^{\mathbb{P}}\end{equation} for every $1\leq i \leq \ell$ and $\delta \in \Delta'$. 
\medskip

\noindent {\em Continuing the proof of Case 1.} Now let us recall that there exists $w_0 \in [N,N]$ such that $\psi(1,w_0)$ is proximal and $V\subset \wedge^m \mathbb{R}^d$ is the subspace spanned by the attracting eigenlines of proximal elements in $\psi(N_{\textup{R}})$. It follows from Claim 1 that $w_0^{\ell!} \in [N,N]\cap \Delta'$. By (\ref{thm2-eq1}) we have $w_0^{\ell!} \in \bigcap_{i=1}^{\ell} \textup{ker}\varepsilon_i$ and hence $\psi(1,w_0^{\ell!})\big|_{V}=\psi(w_0^{\ell!},w_0^{\ell !})\big|_{V}$ is a proximal tranformation of $\mathsf{GL}(V)$ whose attracting fixed point (which is also the attracting fixed point of $\psi(1,w_0)$ in $\mathbb{P}(\wedge^r \mathbb{R}^d)$) lies exactly in one of the limit sets $\big\{\Lambda_{\psi_i}^{\mathbb{P}}\big\}_{i=1}^{\ell}$, say in $\Lambda_{\psi_1}^{\mathbb{P}}$. Let us note that $V$ cannot be one dimensional. If this were the case, since $\Delta'$ has finite abelianization, up to finite-index, its acts trivially on $V$ and hence $\ell_1(\psi(1,w_0))=1$ which is impossible since $\psi(1,w_0)$ is a proximal matrix in $\mathsf{SL}(\wedge^{r}\mathbb{R}^d)$. It follows that $\textup{dim}(V)\geq 2$. 

The representation $\hat{\psi}:\Delta' \rightarrow \mathsf{SL}^{\pm}(V)$, $\hat{\psi}(\delta,\delta)=\psi(\delta,\delta)|_{V}$, is well defined by (\ref{thm2-eq0}), proximal by the choice of $w_0\in [N,N]$. In particular, the image of $\hat{\psi}$ has non-compact closure in $\mathsf{SL}^{\pm}(V)$. It follows by Corlette's superrigidity (see Theorem \ref{superrigidity} and the remark below) that $\hat{\psi}$ is a quasi-isometric embedding and there exist $C_1,c_1>0$, depending on $\psi$, with the property: \begin{equation} \label{thm2-eq2'}\big|\big|\psi (\delta,\delta )|_{V_{1}}\big|\big| \cdot \big|\big|\psi (\delta,\delta)^{-1}|_{V}\big|\big| \geq e^{\frac{1}{\sqrt{\textup{dim}V}}|| \mu(\hat{\psi}(\gamma,\gamma))||_{\mathbb{E}}}\geq  C_{1}e^{c_1|\delta|_{\Delta}} \end{equation} for every $\delta \in \Delta'$. By using (\ref{thm2-eq2'}) for $n\in N$ and $\delta \in \Delta'$ we have the following estimate: \begin{align*}\big|\big|\psi (\delta n,\delta)\big|\big| \cdot \big|\big|\psi (n^{-1}\delta^{-1},\delta^{-1})\big|\big| &\geq \big|\big|\psi (\delta n,\delta)|_{V}\big|\big| \cdot \big|\big|\psi (n^{-1} \delta^{-1},\delta^{-1})|_{V}\big|\big|\\ &\geq \big|\big|\varepsilon_1(n) \psi (\delta,\delta)|_{V}\big|\big| \cdot \big|\big| \varepsilon_1(n)^{-1}\psi (\delta,\delta )^{-1}|_{V}\big|\big|\\  & =\big|\big|\psi (\delta,\delta)|_{V}\big|\big| \cdot \big|\big|\psi (\delta,\delta)^{-1}|_{V}\big|\big|\\ & \geq C_{1}e^{c_1|\delta|_{\Delta}}.\end{align*} Since $\Delta'$ has finite index in $\Delta$, by the previous estimate and Fact \ref{exterior} we conclude that there exist $C_2,c_2>0$ with the property: \begin{equation} \label{thm2-eq3} \big|\big|\rho (\delta n,\delta )\big|\big| \cdot \big|\big|\rho (\delta n,\delta)^{-1}\big|\big| \geq C_{2}e^{c_2|\delta|_{\Delta}} \end{equation} for every $\delta \in \Delta$ and $n \in N$.

\par Let $\tau:\Gamma \rightarrow \Gamma$ be the automorphism of $\Gamma$ swapping the two coordinates. Since $\rho$ is assumed to be semisimple, $\rho \circ \tau$ is also semisimple and by assumption $\rho|_{[N,N]_{\textup{L}}}=(\rho \circ \tau)|_{[N,N]_{\textup{R}}}$ is not distal. By working as previously, we conclude that there exist $C_3,c_3>0$ with the property: \begin{equation} \big|\big|\rho (\delta,\delta n)\big|\big| \cdot \big|\big|\rho (\delta,\delta n)^{-1}\big|\big| \geq C_{3}e^{c_3|\delta|_{\Delta}} \end{equation} for every $\delta \in \Delta$ and $n \in N$. In particular, if $\delta:=n^{-1}$ we have \begin{equation} \label{thm2-eq4} \big|\big|\rho (n,1)\big|\big| \cdot \big|\big|\rho (n,1)^{-1}\big|\big| \geq C_{3}e^{c_3|n|_{\Delta}} \end{equation} for every $n \in N$.
\par Now let $\delta \in \Delta$ and $n \in N$. Since $\Delta$ is finitely generated there exist $C_4,c_4>0$, independent of $\delta \in N$, such that $\big| \big| \rho(\delta,\delta) \big| \big| \cdot \big| \big| \rho(\delta,\delta)^{-1}\big|\big|\leq C_4 e^{c_4|\delta|_{\Delta}}$. Therefore, by using (\ref{thm2-eq4}) we have \begin{equation} \label{thm2-eq5} \big|\big|\rho (\delta n,\delta)\big|\big| \cdot \big|\big|\rho (\delta n,\delta)^{-1}\big|\big| \geq \frac{ \big|\big|\rho (n,1)\big|\big| \cdot \big|\big|\rho (n,1)^{-1}\big|\big|}{ \big|\big|\rho (\delta, \delta)\big|\big| \cdot \big|\big|\rho (\delta,\delta)^{-1}\big|\big|} \geq \frac{C_3}{C_4}e^{c_3|n|_{\Delta}-c_4|\delta|_{\Delta}} \end{equation} for $\delta\in \Delta$ and $n \in N$.

By letting $\theta:=\frac{3c_4}{c_2}$, raising both terms of (\ref{thm2-eq3}) in $\theta>0$ and using (\ref{thm2-eq5}) we have that:  \begin{equation} \big|\big|\rho (\delta n,\delta)\big|\big|^{\theta+1} \cdot \big|\big|\rho (\delta n,\delta)^{-1}\big|\big|^{\theta+1} \geq \frac{C_2^\theta C_3}{C_4} e^{c_2\theta |\delta|_{\Delta}-c_4|\delta|_{\Delta}+c_3|n|_{\Delta}}= \frac{C_2^{\theta} C_3}{C_4} e^{2c_4|\delta|_{\Delta} +c_3|n|_{\Delta}} \end{equation} for every $\delta\in \Delta$ and $n \in N$. In particular, since $|n|_{\Delta}+|\delta|_{\Delta}\geq \frac{1}{2}(|\delta|_{\Delta}+|\delta n|_{\Delta})$, we conclude that there exist $C_5,c_5>0$ such that \begin{equation} \big|\big|\rho (\delta n,\delta)\big|\big| \cdot \big|\big|\rho (\delta n,\delta)^{-1}\big|\big| \geq C_5 e^{c_5 (|n|_{\Delta}+|\delta|_{\Delta})} \geq C_5 e^{\frac{c_5}{2}(|\delta n|_{\Delta}+|\delta|_{\Delta})} \end{equation} for every $\delta\in \Delta$ and $n \in N$. This completes the proof of the theorem when $\rho$ is semisimple. \end{proof}
\medskip

\noindent {\bf Case 2:} {\em Suppose that $\rho:\Gamma \rightarrow \mathsf{SL}_d(\mathbb{R})$ is not a semisimple representation.} 

Let $\rho^{ss}:\Gamma \rightarrow \mathsf{SL}_d(\mathbb{R})$ be a semisimplification of $\rho$. Since $\lambda(\rho^{ss}(g))=\lambda(\rho(g))$ for every $g \in \Gamma$, the restrictions of $\rho^{ss}$ on $[N,N]_{\textup{R}}$ and $[N,N]_{\textup{L}}$ are not distal. It follows by Case 1 and Lemma \ref{qie-semisimple} that there exist $C,c,\delta>0$ such that \begin{align*}\big| \big| \mu (\rho(\delta n,\delta))\big|\big|_{\mathbb{E}}\geq \delta \big| \big| \mu (\rho^{ss}(\delta n,\delta))\big|\big|_{\mathbb{E}} &\geq \frac{1}{\sqrt{2}} \log \big(\big|\big|\rho^{ss}(\delta,\delta n)\big|\big|\cdot \big|\big|\rho^{ss}(\delta,\delta n)^{-1}\big|\big| \big) \\ &\geq c\big(|\delta n|_{\Delta}+|n|_{\Delta}\big)-C\end{align*} for every $\delta \in \Delta$ and $n\in N$. The proof is complete.  \end{proof}

\begin{proof}[Proof of Theorem \ref{main1}.] For the implication $\textup{(iii)} \Rightarrow \textup{(ii)}$ note that the kernel of $\rho$ is torsion subgroup of $\Delta \times_N \Delta$ and hence of $\Delta \times \Delta$. Since torsion subgroups of hyperbolic groups are finite \cite{Gromov} the same holds for the direct product $\Delta \times \Delta$, hence $\textup{ker}\rho$ is finite.

The implication $\textup{(ii)}\Rightarrow \textup{(i)}$ follows by Lemma \ref{non-distal}. Now we prove $\textup{(i)}\Rightarrow \textup{(iii)}$. Suppose that $\rho:\Delta \times_N \Delta \rightarrow \mathsf{SL}_d(\mathbb{R})$ is a representation such that the restrictions of $\rho$ on $[N,N]_{\textup{R}}$ and $[N,N]_{\textup{L}}$ are not distal. Since the quotient group $\Delta/N$ is hyperbolic, by Theorem \ref{main2} and Proposition \ref{undist} there exist constants $C,C',c,c'>0$ such that \begin{align*} \big| \big| \mu (\rho(\delta n,\delta))\big|\big|_{\mathbb{E}} \geq c\big(|\delta n|_{\Delta}+|n|_{\Delta}\big)-C \geq c'\big|(\delta n,\delta)\big|_{\Delta \times_N \Delta}-C' \end{align*} for every $\delta \in \Delta$ and $n\in N$. In particular, $\rho$ is a quasi-isometric embedding.  Hence $\textup{(i)}\Rightarrow \textup{(iii)}$ follows.\end{proof}

\section{Further properties of the examples} \label{add}

In this section we provide some further properties of the examples constructed in this paper. For a finitely generated group $\mathsf{H}$ its first Betti number, denoted by $b_1(\mathsf{H})$, is the free rank of the finitely generated abelian group $\mathsf{H}/[\mathsf{H},\mathsf{H}]$. 

First we explain that there exist infinitely many isomorphism classes of examples in $\Delta \times \Delta$ with positive first Betti number satisfying the conclusion of Theorem \ref{main1}. 

\begin{proposition} \label{main3} Let $\Delta$ be a torsion-free cocompact lattice in $\mathsf{Sp}(m,1)$, $m \geq 2$, or $\textup{F}_4^{(-20)}$. There exist infinitely many non-isomorphic subgroups $P$ of $\Delta \times \Delta$ with positive first Betti number such that for $d \in \mathbb{N}$ every discrete faithful representation of $P$ into $\mathsf{SL}_d(\mathbb{R})$ is a quasi-isometric embedding. \end{proposition}

The first key tool that we need is the following theorem about the existence of quotients of hyperbolic groups which are non-elementary hyperbolic. Let us recall that for a group $\Gamma$ and a finite subset $\mathcal{F}$ of $\Gamma$ we set $\llangle \mathcal{F}\rrangle =\big \langle \{ gfg^{-1}:g \in \Gamma, f \in \mathcal{F}\}\big\rangle$.

\begin{theorem} \textup{(}\cite{Gromov, Ol,Delzant}\textup{)}\label{quotient} Let $\Gamma$ be a non-elementary hyperbolic group and $w \in \Gamma$ be an infinite order element. Then the quotient $\Gamma/\llangle w^m \rrangle$ is non-elementary hyperbolic for all but finitely many $m\in \mathbb{N}$.\end{theorem}

Let $\Lambda$ be a group, $L$ be a subgroup of $\Lambda$ and $N$ be a normal subgroup of $L$. The triple $(\Lambda,L,N)$ satisfies the {\em Cohen--Lyndon property} if there exists a set $T$ of left coset representatives of $\llangle N\rrangle L$ in $\Lambda$ such that $$ \llangle N \rrangle= \big \langle \{ tNt^{-1}:t \in T\} \big\rangle= \bigast_{t\in T}tNt^{-1}.$$ Sun in \cite{Sun} established a Cohen--Lyndon type theorem for any group $\Lambda$ and any hyperbolically embedded subgroup $L$ of $\Lambda$. We need the following special case of \cite[Thm. 2.5]{Sun} for maximal cyclic subgroups of torsion-free hyperbolic groups.

\begin{theorem} \textup{(}\cite{Sun}\textup{)} \label{LC} Let $\Gamma$ be a non-elementary torsion-free word hyperbolic group and $\langle w \rangle$ be an infinite maximal cyclic subgroup of $\Gamma$. Then for all but finitely many $n\in \mathbb{N}$ the triple $(\Gamma, \langle w \rangle, \langle w^n \rangle )$ has the Cohen--Lyndon property. \end{theorem}

Mj--Mondal in \cite{MM} proved the following proposition in order to establish sufficient conditions so that certain fiber products do not have Property (T).

\begin{proposition} \textup{(}\cite[Prop. 3.6]{MM}\textup{)} \label{MM} Let $\Lambda$ be a group, $L$ be a subgroup of $\Lambda$ and $N$ be a normal subgroup of $L$. Suppose that the triple $(\Lambda,L,N)$ satisfies the Cohen--Lyndon property. Then there exists a surjective group homomorphism $$\phi: \llangle N  \rrangle \slash [ \llangle N \rrangle, \Lambda] \twoheadrightarrow N/[L,N].$$\end{proposition}

We will need the following consequence of the compactness theorem, see \cite[p. 340]{Paulin} and \cite[Thm. 3.9]{Bestvina}, and the fact that every isometric action of a group with Property (T) on a real tree has a globally fixed point.

\begin{proposition}\textup{(\cite{Paulin, Bestvina})} Let $\Gamma_1$ be a finitely generated group with Property (T) and $\Gamma_2$ be a hyperbolic group. Suppose that $\big\{\varphi_n:\Gamma_1\rightarrow \Gamma_2\big\}_{n\in \mathbb{N}}$ is a sequence of group homomorphisms. There exists $r\in \mathbb{N}$, a subsequence $(\varphi_{m_n})_{n\in \mathbb{N}}$ and $(\gamma_n)_{n\in \mathbb{N}}\subset \Gamma_2$ \hbox{such that for every $\delta\in \Gamma_1$ and $n\in \mathbb{N}$:} $$\varphi_{m_n}(\delta)=\gamma_n \varphi_{r}(\delta)\gamma_n^{-1}.$$ \end{proposition}

Now we can give the proof of Proposition \ref{main3}.

\begin{proof}[Proof of Proposition \ref{main3}.] Let $\langle f_1 \rangle$ be an infinite maximal cyclic subgroup of $\Delta$. By Theorem \ref{quotient} and Theorem \ref{LC} we may choose $k_1 \in \mathbb{N}$ such that $\Delta/ \langle \langle f_1^{k_1} \rangle \rangle$ is non-elementary hyperbolic and $(\Delta, \langle f_1 \rangle, \langle f_1^{k_1} \rangle)$ has the Cohen--Lyndon property. Observe that the quotient of $\Delta \times _{\llangle f_1^{k_1} \rrangle} \Delta$ by the normal subgroup $\Delta \times _{[\Delta,\llangle f_1^{k_1} \rrangle]} \Delta$ is isomorphic to $\llangle f_1^{k_1} \rrangle/ [\Delta,\llangle f_1^{k_1} \rrangle]$ and hence by Proposition \ref{MM} we have $b_1\big(\Delta \times_{\llangle f_1^{k_1} \rrangle} \Delta\big)>0$. Let $N_1:=\llangle f_1^{k_1} \rrangle$ and $\Delta_1:=\Delta/\llangle f_1^{k_1} \rrangle$. Next, we choose $\langle f_2 N_1 \rangle$ a maximal cyclic infinite subgroup of $\Delta/N_1$. Note that $\langle f_2 \rangle$ is an infinite maximal cyclic subgroup of $\Delta$, hence by Theorem \ref{quotient}, Theorem \ref{LC} and Proposition \ref{MM} we may choose $k_2 \in \mathbb{N}$ such that $\Delta_2:=\Delta/\langle \langle f_2^{k_2} \rangle \rangle$ and $\Delta/N_2$, $N_2:=\llangle f_1^{k_1},f_2^{k_2} \rrangle$, are non-elementary hyperbolic and $b_1\big(\Delta \times_{\llangle f_2^{k_2} \rrangle} \Delta\big)>0$. By continuing similarly, we obtain a sequence of elements $(f_q)_{q \in \mathbb{N}}$ of $\Delta$ and integers $(k_q)_{q \in \mathbb{N}}$ such that:

\medskip

\noindent \textup{(i)} For every $q \in \mathbb{N}$, the quotient $\Delta/N_q$, $N_{q}= \llangle f_{1}^{k_1},\dots, f_{q}^{k_q} \rrangle$, is non-elementary hyperbolic.

\noindent \textup{(ii)} For $q<p$, $\langle f_{p}N_q \rangle$ is an infinite maximal cyclic subgroup of $\Delta/N_{q}$. In particular, $\langle f_q \rangle$ is a maximal cyclic subgroup of $\Delta$.

\noindent \textup{(iii)} For every $q \in \mathbb{N}$, $\Delta_q:=\Delta/\llangle f_{q}^{k_q} \rrangle$ is a non-elementary hyperbolic group and \hbox{$b_1 \big(\Delta \times_{\llangle f_q^{k_q}\rrangle}\Delta \big)>0$.}

\medskip
We claim that for every $q_0\in \mathbb{N}$ there exist finitely many $q\in \mathbb{N}$ such that $\Delta_q$ is isomorphic to $\Delta_{q_0}$. Suppose that this does not happen, i.e. there exists an infinite sequence $(s_q)_{q \in \mathbb{N}}$ and isomorphisms $\phi_{s_q}:\Delta_{s_q} \rightarrow \Delta_{q_0}$. Let $\pi_{s_q}:\Delta \twoheadrightarrow \Delta_{s_q}$ be the projection with kernel $\llangle f_{s_q}^{k_{s_q}} \rrangle$. In particular, we obtain a sequence of surjective group homomorphisms $\phi_{s_q}\circ \pi_{s_q}: \Delta \twoheadrightarrow \Delta_{q_0}$. 

Let us set $\mathsf{N}_{s_q}:=\llangle f_{s_q}^{k_{s_q}}\rrangle$ for $q \in \mathbb{N}$. It follows by the previous proposition that, up to passing to a subsequence, there exists a sequence $(g_q)_{q\in \mathbb{N}}$ of elements in $\Delta_{q_0}$ and $r \in \mathbb{N}$ such that $$\phi_{s_q}(\pi_{s_q}(g))=g_q \phi_{s_{r}}(\pi_{s_r}(g))g_{q}^{-1}$$ for every $q \in \mathbb{N}$ and $g \in \Delta$. In particular, since $\phi_{s_r}$ is injective, for $g:=f_{s_q}^{k_{s_q}}$ and $q\in \mathbb{N}$ large enough, we have $\pi_{s_r}(f_{s_q}^{k_{s_q}})=1$ or equivalently $f_{s_q}^{k_{s_q}}\in \llangle f_{s_r}^{k_{s_r}}\rrangle$. This is a contradiction since by construction (see (iii)) $f_{s_q}N_{s_r}$ generates an infinite maximal cyclic subgroup of $\Delta/\mathsf{N}_{s_r}$ for $q\in \mathbb{N}$ large enough. Therefore, we may pass to a subsequence, still denoted $\big\{\Delta_{s_q}\big \}_{q\in \mathbb{N}}$, such that $\Delta_{s_q}$ is not isomorphic to $\Delta_{s_p}$ \hbox{for $p \neq q$.}
\par By construction, for every $q \in \mathbb{N}$, $\Delta \times_{\mathsf{N}_{s_q}}\Delta$ has positive first Betti number and satisfies the conclusion of Theorem \ref{main1} since $\Delta/\mathsf{N}_{s_q}$ is hyperbolic. Now observe that  in the fiber product $\Delta \times_{\mathsf{N}_{s_p}}\Delta$ the subgroups $\{1\}\times \mathsf{N}_{s_p}$ and $ \mathsf{N}_{s_p}\times \{1\}$ are the only non abelian centralizers of non-cyclic non-trivial subgroups of $\Delta \times_{\mathsf{N}_{s_p}}\Delta$ (see \cite[\S 6]{Bridson-Grunewald}). Now suppose that there exists a group isomorphism $f_{q,p}:\Delta \times_{\mathsf{N}_{s_q}}\Delta \rightarrow \Delta \times_{\mathsf{N}_{s_p}}\Delta$. The previous observation shows $f_{q,p}(\mathsf{N}_{s_q}\times \mathsf{N}_{s_q})=\mathsf{N}_{s_p}\times \mathsf{N}_{s_p}$ \hbox{and since} $$\Delta \times_{\mathsf{N}_{s_i}} \Delta/\mathsf{N}_{s_i}\times \mathsf{N}_{s_i} \cong \Delta/\mathsf{N}_{s_i}\ i\in \{p,q\},$$ $f_{q,p}$ induces an isomorphism $f_{q,p}':\Delta/\mathsf{N}_{s_q} \rightarrow \Delta/\mathsf{N}_{s_p}$. Therefore, $p=q$. It follows that $\big\{\Delta \times_{\mathsf{N}_{s_p}}\Delta \big\}_{p \in \mathbb{N}}$ is an infinite sequence of pairwise non-isomorphic subgroups of $\Delta \times \Delta$.\end{proof}

We close this section by showing that that the fiber product of a hyperbolic group with respect to an infinite index infinite normal subgroup is not commensurable to a lattice in any semisimple group. The proof follows the strategy of the proof in \cite[Thm. 1.4 (d)]{Bass-Lubotzky} (see p. 1171--1172), however with certain modifications since $\mathsf{\Gamma}$ is not assumed to be a superrigid rank 1 lattice.

\begin{proposition} \label{nonlattice} Let $\mathsf{\Gamma}$ be a virtually torsion-free hyperbolic group and $\mathsf{N}$ be an infinite normal subgroup of $\mathsf{\Gamma}$ of infinite index. Suppose that $G_1,\ldots, G_{\ell}$ are connected simple algebraic groups defined over the local field $k_1,\ldots,k_{\ell}$ respectively. The fiber product $\mathsf{\Gamma}\times_{\mathsf{N}}\mathsf{\Gamma}$ is not commensurable to a lattice in the locally compact group ${\bf G}=G_1(k_1)\times \cdots \times G_{\ell}(k_{\ell})$.\end{proposition}

\begin{proof} Suppose that $Q$ is a finite-index subgroup of $\mathsf{\Gamma}\times_{\mathsf{N}}\mathsf{\Gamma}$ which is a lattice in ${\bf G}$. Note that $Q$ contains a finite-index subgroup $Q_1$ which is normal in $\mathsf{\Gamma}\times_{\mathsf{N}}\mathsf{\Gamma}$. The intersection $Q_1\cap \textup{diag}(\mathsf{\Gamma}\times \mathsf{\Gamma})$ is a finite-index normal subgroup of $\textup{diag}(\mathsf{\Gamma}\times \mathsf{\Gamma})$ and is of the form $\textup{diag}(\mathsf{\Gamma}_1\times \mathsf{\Gamma}_1)$ for some finite-index normal subgroup $\mathsf{\Gamma}_1$ of $\mathsf{\Gamma}$. Thus $Q$ contains a finite-index subgroup of the form $\mathsf{\Gamma}_1\times_{\mathsf{N}_1}\mathsf{\Gamma}_1$, where $\mathsf{N}_1=\mathsf{N}\cap\mathsf{\Gamma}_1$ is of finite index in $\mathsf{N}$.

By the previous observation, without loss of generality, we may assume that $Q=\mathsf{\Gamma}\times_{\mathsf{N}}\mathsf{\Gamma}$ and that $\mathsf{\Gamma}$ is torsion-free. Let $\textup{pr}_i: Q\rightarrow G_i(k_i)$ denote the projection to the $i$-th coordinate for $1\leq i \leq \ell$. We may also assume that $\textup{pr}_i(Q)$ is not relatively compact. By Borel's density theorem \cite{Borel} (see also \cite[Cor. 3.2]{Dani}) the projection $\textup{pr}_i(Q)$ is Zariski dense in $G_i(k_i)$. Observe that the normalizer of the Zariski closure of $\textup{pr}_i(\mathsf{N}_{\textup{L}})$ (and $\textup{pr}_i(\mathsf{N}_{\textup{R}})$) in $G_i(k_i)$ is algebraic. Since $\textup{pr}_i(\mathsf{N}_{\textup{L}})$ and $\textup{pr}_i(\mathsf{N}_{\textup{R}})$ commute, either $\textup{pr}_i(\mathsf{N}_{\textup{L}})$ or $\textup{pr}_i(\mathsf{N}_{\textup{R}})$ is central. Moreover, up to passing to a finite index subgroup of $\mathsf{N}$, we may assume that for every $1 \leq i \leq \ell$ either $\textup{pr}_i(\mathsf{N}_{\textup{L}})$ or $\textup{pr}_i(\mathsf{N}_{\textup{R}})$ are trivial.

It follows that $\ell \geq 2$ and let ${\bf G_1}$ (resp. ${\bf G_2}$) be the product of the $G_i(k_i)$ such that $\textup{pr}_i(\mathsf{N}_{\textup{L}})$ (resp. $\textup{pr}_i(\mathsf{N}_{\textup{R}})$) is trivial. In particular, we obtain a discrete faithful representation $\rho:Q\xhookrightarrow{} {\bf G_1}\times {\bf G_2}$, $$\rho(\gamma, \gamma n)=\big(\rho_1(\gamma, \gamma n),\rho_2(\gamma, \gamma n)\big),(\gamma, \gamma n)\in Q,$$ such that $\rho(Q)$ is a lattice. Let us observe that the restriction of $\rho_1$ on $\textup{diag}(\mathsf{\Gamma}\times \mathsf{\Gamma})$ is faithful. Indeed, $\rho_1(\gamma,\gamma)$ is trivial for some $\gamma \in \Gamma$, then for every $n\in \mathsf{N}$ we have $\rho(\gamma, \gamma)\rho(1,n) \rho(\gamma, \gamma)^{-1}=\rho(1,n)$ since $\rho_2(1,n)$ is trivial by the definition of ${\bf G_2}$. Note that $\rho$ is faithful and hence $\gamma n\gamma^{-1}=n$. It follows that $\gamma$ is trivial since the centralizer of $\mathsf{N}$ in $\mathsf{\Gamma}$ is trivial.
\par Now let $\mathsf{H}=\rho_1(\textup{diag}(\mathsf{\Gamma \times \Gamma}))\times \rho_2(\textup{diag}(\mathsf{\Gamma \times \Gamma}))$ and observe that $\mathsf{H}$ is a discrete subgroup of ${\bf G_1}\times {\bf G_2}$. Indeed, if $(\gamma_r)_{r\in \mathbb{N}}$ and $(\delta_r)_{r \in \mathbb{N}}$ are sequences in $\mathsf{\Gamma}$ such that the sequence $g_{r}:=\big(\rho_1(\gamma_r,\gamma_r),\rho_2(\delta_r,\delta_r)\big)$ converges to $(1,1)$, then $\lim_{r } g_r \rho(n_1,n_2)g_{r}^{-1}=\lim_{r}\rho(\gamma_r n_1 \gamma_r^{-1}, \delta_r n_2 \delta_r^{-1})=\rho(n_1,n_2).$ Since $\rho$ is discrete, for all but finitely many $r \in \mathbb{N}$, $\gamma_r$ (resp. $\delta_r$) centralizes $n_1\in \mathsf{N}$ (resp. $n_2\in \mathsf{N}$). Recal that (since $\mathsf{\Gamma}$ is torsion-free) the centralizer of a maximal cyclic subgroup of $\mathsf{\Gamma}$ is cyclic and since $n_1,n_2\in \mathsf{N}$ were arbitrary and $\mathsf{N}$ is not cyclic, $(\gamma_r)_{r \in \mathbb{N}}$ and $(\delta_r)_{r \in \mathbb{N}}$ have to be eventually trivial. It follows that $\mathsf{H}$ is a discrete subgroup of ${\bf G_1}\times {\bf G_2}$. Since $\rho(Q)$ is assumed to be a lattice in ${\bf G_1}\times {\bf G_2}$, it has finite index in $\mathsf{H}$. In particular, there exists a finite index subgroup $\mathsf{\Gamma}'$ of $\mathsf{\Gamma}$ such that $\rho_1(\textup{diag}(\mathsf{\Gamma}'\times \Gamma'))\times \{1\}$ is a subgroup of $\rho(Q)$ centralizing the group $\rho(\mathsf{N}_{\textup{L}})$. The centralizer of $\rho(\mathsf{N}_{\textup{L}})$ in $\rho(Q)$ is $\rho(\mathsf{N}_{\textup{R}})$, hence for every $\gamma'\in \mathsf{\Gamma}'$ there exists $n'\in \mathsf{N}$ with $\rho_1(\gamma',\gamma')=\rho_1(1,n)=\rho_1(n,n)$, so $\gamma'=n$ since $\rho_1|_{\textup{diag}(\mathsf{\Gamma}'\times \Gamma')}$ is faithful. This contradicts the fact that $\mathsf{N}$ has infinite index in $\mathsf{\Gamma}$.\end{proof}

\end{document}